\documentclass[10pt]{amsart}
\usepackage{mathrsfs}
\usepackage{amssymb}
\usepackage[all]{xy}
\usepackage{amsthm}
\usepackage{array}
\usepackage{amsmath}
\usepackage{enumitem}

\newtheorem{theorem}{Theorem}[section]
\newtheorem{lemma}[theorem]{Lemma}
\newtheorem{proposition}[theorem]{Proposition}
\newtheorem{corollary}[theorem]{Corollary}
\newtheorem{remark}[theorem]{Remark}

\newtheorem{definition}[theorem]{Definition}

\def\be{\begin{equation}}
\def\ee{\end{equation}}
\def\br{\begin{eqnarray}}
\def\er{\end{eqnarray}}

 \title{Morse-Novikov cohomology on complex manifolds}
 \author{Lingxu Meng}

\address{Department of Mathematics\\  North University of China \\ Taiyuan, Shanxi 030051,
P.R. China}
\email{menglingxu@nuc.edu.cn}
\thanks{}
\date{}

\begin{document}
\maketitle

\begin{abstract}
We view Dolbeault-Morse-Novikov cohomology $H^{p,q}_\eta(X)$ as the cohomology of the sheaf $\Omega_{X,\eta}^p$ of $\eta$-holomorphic $p$-forms  and give several bimeromorphic invariants. Analogue to Dolbeault cohomology, we establish the Leray-Hirsch theorem and the blow-up formula for Dolbeault-Morse-Novikov cohomology. At last, we consider the relations between Morse-Novikov cohomology and Dolbeault-Morse-Novikov cohomology, moreover, investigate  stabilities of their dimensions under the deformations of complex structures. In some aspects, Morse-Novikov and Dolbeault-Morse-Novikov cohomology behave similarly with de Rham and Dolbeault cohomology.
\end{abstract}


\textbf{AMSC}: 32C35, 57R19.

\section{Introduction}
Let $X$ be a  smooth manifold and  $\theta$ a real closed $1$-form on $X$. Set $\mathcal{A}^p(X)$ the space of real smooth $p$-forms and define $\textrm{d}_\theta:\mathcal{A}^p(X)\rightarrow\mathcal{A}^{p+1}(X)$ as $\textrm{d}_\theta\alpha=\textrm{d}\alpha+\theta\wedge\alpha$ for $\alpha\in\mathcal{A}^p(X)$.
Clearly, $\textrm{d}_\theta\circ \textrm{d}_\theta=0$, so we have a complex
\begin{displaymath}
\xymatrix{
\cdots\ar[r] &\mathcal{A}^{p-1}(X)\ar[r]^{\textrm{d}_\theta}&\mathcal{A}^{p}(X)\ar[r]^{\textrm{d}_\theta\quad}&\mathcal{A}^{p+1}(X)\cdots\ar[r]&\cdots
},
\end{displaymath}
whose cohomology $H^p_\theta(X)=H^p(\mathcal{A}^\bullet(X),\textrm{d}_\theta)$ is called the \emph{$p$-th Morse-Novikov cohomology}. For a complex closed $1$-form $\theta$ on $X$, denote $H^p_\theta(X,\mathbb{C})=H^p(\mathcal{A}_{\mathbb{C}}^\bullet(X),\textrm{d}_\theta)$, where $\mathcal{A}_{\mathbb{C}}^\bullet(X)=\mathcal{A}^\bullet(X)\otimes_{\mathbb{R}}\mathbb{C}$. If $\theta$ is real, $H^p_\theta(X,\mathbb{C})=H^p_\theta(X)\otimes_{\mathbb{R}}\mathbb{C}$. Similarly, we can define \emph{Morse-Novikov cohomology with compact support} $H^p_{\theta,c}(X)$ and $H^p_{\theta,c}(X,\mathbb{C})$.

This cohomology was originally defined by A. Lichnerowicz (\cite{L}) and D. Sullivan (\cite{S}) in the context of Poisson geometry and infinitesimal computations in topology, respectively. It is well used to study the locally conformally K\"{a}hlerian (l.c.K.) and locally conformally symplectic (l.c.s.) structures (\cite{Ba1,Ba2,BK,HR,LLMP,Vai}).  $H_{\theta}^*(X)$ can be viewed as the cohomology of  a flat bundle (weight line bundle) or a local constant sheaf of $\mathbb{R}$-modules with finite rank, referring to \cite{S, M1, N, OV1, YZ}. As we know, the two viewpoints are equivalent, whereas the latter is much more convenient, seeing \cite{M1}.

In his seminal paper \cite{N}, S. P. Novikov introduced a generalization of the classical Morse theory to the case of circle-valued Morse functions. A. Pajitnov \cite{P1}  observed the relation of the circle-valued Morse theory to the homology with local coefficients and perturbed de Rham differential, see also \cite{P2}, p. 414-416.

For smooth manfiolds, the Mayer-Vietoris sequence and Poincar\'{e} duality theorem were generalized on Morse-Novikov cohomology by  S. Haller and T. Rybicki  \cite{HR}. M. Le\'{o}n, B. L\'{o}pez,  J. C. Marrero  and E. Padr\'{o}n \cite{LLMP} proved that a compact Riemannian manifold $X$ endowed with a parallel one-form $\theta$ has trivial  Morse-Novikov cohomology.  By Atiyah-Singer index theorem, G.  Bande and D. Kotschick  \cite{BK} found that the Euler characteristic of Morse-Novikov cohomology coincides with the usual Euler characteristic. In \cite{M1}, we proved several K\"{u}nneth formulas and theorems of Leray-Hirsch type.

For complex manifolds, I. Vaisman \cite{Vai} studied the classical operators twisted with a closed one-form on l.c.K. manifolds. In \cite{M1}, we gave two explicit formulas of blow-ups of complex manifolds for Morse-Novikov cohomology.  As we know, de Rham cohomology is closely related to Dolbeault cohomology on complex manifolds, such as Hodge decomposition  theorem,  hard Lefschetz theorem, Hodge's index theorem, etc.. Inspired by these, it is necessary to study Dolbeault-Morse-Novikov cohomology, which is a generalization of Dolbeault cohomology. Recently, L. Ornea, M. Verbitsky,  and V. Vuletescu \cite{OVV2} showed that, for a locally conformally K\"{a}hler manifold $X$ with proper potential, $H_{a\eta}^{*,*}(X)=0$ holds for all $a\in \mathbb{C}$ but a discrete countable subset, where $\eta$ is the $(0,1)$-part of Lee form $\theta$ of $X$.

L. Ornea, M. Verbitsky,  and V. Vuletescu \cite{OVV1} proved that the blow-up of an l.c.K. manifold along a submanifold is l.c.K. if and only if the submanifold is globally conformally equivalent to a K\"{a}hler submanifold. So, it is necessary to consider the variance of the Morse-Novikov  (\cite{M1}) and Dolbeault-Morse-Novikov cohomology under blowing up.
\begin{theorem}\label{main1}
Let $\pi:\widetilde{X}\rightarrow X$ be the blow-up of a connected complex manifold $X$ along a connected complex submanifold $Z$ and $i_E:E=\pi^{-1}(Z)\rightarrow\widetilde{X}$ the inclusion of the exceptional divisor $E$ into $\widetilde{X}$. Suppose that $\eta$ is a $\bar{\partial}$-closed $(0,1)$-form on $X$  and $\tilde{\eta}=\pi^*\eta$. Then, for any $p$, $q$,
\begin{displaymath}
\pi^*+\sum_{i=0}^{r-2}(i_E)_*\circ (h^i\cup)\circ (\pi|_E)^*
\end{displaymath}
gives an isomorphism
\begin{equation}\label{iso}
H_{\eta}^{p,q}(X)\oplus \bigoplus_{i=0}^{r-2}H_{\eta|_Z}^{p-1-i,q-1-i}(Z)\tilde{\rightarrow} H_{\tilde{\eta}}^{p,q}(\widetilde{X}),
\end{equation}
where $r=\emph{codim}_{\mathbb{C}}Z$ and $h$ is defined in \emph{(\ref{chern})}.
\end{theorem}
For $\eta=0$, S. Rao, S. Yang, and X.-D. Yang \cite{RYY} first proved there exists an isomorphism (\ref{iso}) on a compact complex manifold $X$. It seems difficult to write out  it explicitly using their method. In \cite{M2},  we write out an isomorphism explicitly on any (possibly noncompact) base with a different way.

Deformations of complex structures play a significant role in studying K\"{a}hlerian, balanced, strongly Gauduchon and $\partial\overline{\partial}$-manifolds. For l.c.K. geometry, we have known the facts that a deformation of a l.c.K. manifold is generally not l.c.K. (\cite{Be}) and the class of compact l.c.K. manifolds with potential is stable under small deformations (\cite{OV2}). These results inspire us to investigate  behaviors of Dolbeault-Morse-Novikov cohomology under deformations.
\begin{lemma}\label{main2}
Let $f:X\rightarrow Y$ be a proper surjective submersion of connected smooth manifolds and $\theta$ a real $($resp. complex$)$ closed $1$-form on $X$. Then, for any $k$, the higher direct image $R^kf_*\underline{\mathbb{R}}_{X,\theta}$  $($resp. $R^kf_*\underline{\mathbb{C}}_{X,\theta}$$)$ is a local system of $\mathbb{R}$ $($resp. $\mathbb{C}$$)$-modules with finite rank.
\end{lemma}
Using above lemma and the relation between Morse-Novikov and Dolbeault-Morse-Novikov cohomologies, we get the theorem of stability of $\eta$-hodge numbers under the deformation.
\begin{theorem}\label{1.3}
Let $f:X\rightarrow Y$ be a family of complex manifolds and $\theta$ a complex closed $1$-form on $X$. Assume $b_k(X_o,\theta|_{X_o})=\sum_{p+q=k}h_{\eta|_{X_o}}^{p,q}(X_o)$ for some $k$ and some point $o\in Y$, where $\eta$ is the $(0,1)$-part of $\theta$. Then, for any $t$ near $o$, $h_{\eta|_{X_t}}^{p,q}(X_t)=h_{\eta|_{X_o}}^{p,q}(X_o)$, where $\eta$ is the $(0,1)$-part of $\theta$ and $p+q=k$.
\end{theorem}

In this article, we investigate the Dolbeault-Morse-Novikov cohomology via the theory of sheaves. In Sec. 2 and 3, we recall the Morse-Novikov cohomology and define the Dolbeault-Morse-Novikov cohomology, respectively. In Sec. 4, we study the properties of the sheaf $\mathcal{O}_{X,\eta}$ of $\eta$-holomorphic functions  and show that  $H_\eta^{p,0}(X)$, $H_{\eta,c}^{p,0}(X)$, $H_\eta^{0,p}(X)$ and $H_{\eta,c}^{0,p}(X)$ are all bimeromorphic invariants. In particular, we prove Leray-Hirsch theorem and Theorem \ref{main1}.  In Sec. 5,  Lemma \ref{main2} and Theorem \ref{1.3} are proved.

\section{Morse-Novikov cohomology}
We first recall the weight $\theta$-sheaf, refering to \cite{M1}. Let $\mathcal{A}_X^k$ be the sheaf of germs of real smooth $k$-forms and $\underline{\mathbb{R}}_{X}$, $\underline{\mathbb{C}}_{X}$ be constant sheaves with coefficient $\mathbb{R}$, $\mathbb{C}$  on $X$, respectively. Set $\mathcal{A}_{X,\mathbb{C}}^k=\mathcal{A}_{X}^k\otimes_{\underline{\mathbb{R}}_{X}}\underline{\mathbb{C}}_{X}$. Define $\textrm{d}_\theta:\mathcal{A}_{X,\mathbb{C}}^k\rightarrow\mathcal{A}_{X,\mathbb{C}}^{k+1}$ as $\textrm{d}_\theta\alpha=\textrm{d}\alpha+\theta\wedge\alpha$, for $\alpha\in\mathcal{A}_{X,\mathbb{C}}^k$.
\begin{definition}
The kernel of $\emph{d}_\theta:\mathcal{A}_{X,\mathbb{C}}^0\rightarrow\mathcal{A}_{X,\mathbb{C}}^1$ is called a weight $\theta$-sheaf, denoted by $\underline{\mathbb{C}}_{X,\theta}$.
\end{definition}
Locally, $\theta=\textrm{d}u$ for a smooth complex-valued function $u$, so $\textrm{d}_\theta=e^{-u}\circ \textrm{d}\circ e^u$ and $\underline{\mathbb{C}}_{X,\theta}=\mathbb{C}e^{-u}$. Hence, the weight $\theta$-sheaf $\underline{\mathbb{C}}_{X,\theta}$ is a local system of $\mathbb{C}$-modules with rank 1. We have a resolution of soft sheaves of $\underline{\mathbb{C}}_{X,\theta}$
\begin{displaymath}
\xymatrix{
0\ar[r] &\underline{\mathbb{C}}_{X,\theta}\ar[r]^{i} &\mathcal{A}_{X,\mathbb{C}}^0\ar[r]^{\textrm{d}_\theta} &\mathcal{A}_{X,\mathbb{C}}^1\ar[r]^{\textrm{d}_\theta}&\cdots\ar[r]^{\textrm{d}_\theta}&\mathcal{A}_{X,\mathbb{C}}^n\ar[r]&0
},
\end{displaymath}
where $i$ is the natural inclusion. So
\begin{displaymath}
H_\theta^*(X,\mathbb{C})\cong H^*(X,\underline{\mathbb{C}}_{X,\theta}),\quad H_{\theta,c}^*(X,\mathbb{C})\cong H_c^*(X,\underline{\mathbb{C}}_{X,\theta}).
\end{displaymath}
For $\textrm{d}_\theta$-closed $\alpha\in\mathcal{A}_{\mathbb{C}}^*(X)$, denote by $[\alpha]_{\theta}$ (resp. $[\alpha]_{\theta,c}$) the class in $H_\theta^*(X,\mathbb{C})$ (resp. $H_{\theta,c}^*(X,\mathbb{C})$).

Assume $X$ is also oriented. Let $\mathcal{D}_X^{\prime k}$ be the sheaf of germs of real $k$-currents and $\mathcal{D}_{X,\mathbb{C}}^{\prime k}=\mathcal{D}_X^{\prime k}\otimes_{\underline{\mathbb{R}}_X}\underline{\mathbb{C}}_X$. Similarly, define $\textrm{d}_\theta:\mathcal{D}_{X,\mathbb{C}}^{\prime k}\rightarrow\mathcal{D}_{X,\mathbb{C}}^{\prime k+1}$ as $\textrm{d}_\theta T=\textrm{d}T+\theta\wedge T$ for $T\in\mathcal{D}_{X,\mathbb{C}}^{\prime k}$. We have another resolution
\begin{displaymath}
\xymatrix{
0\ar[r] &\underline{\mathbb{C}}_{X,\theta}\ar[r]^{i} &\mathcal{D}_{X,\mathbb{C}}^{\prime 0}\ar[r]^{\textrm{d}_\theta} &\mathcal{D}_{X,\mathbb{C}}^{\prime 1}\ar[r]^{\textrm{d}_\theta}&\cdots\ar[r]^{\textrm{d}_\theta}&\mathcal{D}_{X,\mathbb{C}}^{\prime n}\ar[r]&0
},
\end{displaymath}
of soft sheaves of $\underline{\mathbb{C}}_{X,\theta}$, where $i$ is the natural inclusion. By \cite{Dem}, p. 213 (6.3) (6.4)  and p. 217 (7.8), the natural morphism $\mathcal{A}_{X,\mathbb{C}}^\bullet\hookrightarrow\mathcal{D}_{X,\mathbb{C}}^{\prime\bullet}$ of resolutions induces isomorphisms
\begin{displaymath}
H_\theta^*(X,\mathbb{C})\tilde{\rightarrow} H^*(\mathcal{D}^{\prime\bullet}_{\mathbb{C}}(X),\textrm{d}_{\theta}), \quad H_{\theta,c}^*(X,\mathbb{C})\tilde{\rightarrow} H^*(\mathcal{D}_{\mathbb{C},c}^{\prime\bullet}(X),\textrm{d}_{\theta}).
\end{displaymath}
For $\textrm{d}_\theta$-closed $T\in\mathcal{D}_\mathbb{C}^{\prime*}(X)$, denote by $[T]_{\theta}$ (resp. $[T]_{\theta,c}$ ) the class in $H_\theta^*(X,\mathbb{C})$ (resp. $H_{\theta,c}^*(X,\mathbb{C})$).

\begin{lemma}[\cite{M1}]\label{lem fun 1}
Let $X$ be a connected smooth manifold and $\theta$ a complex closed $1$-form on $X$.

$(1)$ $\underline{\mathbb{C}}_{X,\theta}\cong\underline{\mathbb{C}}_{X}$ if and only if $\theta$ is exact. More precisely, if $\theta=\emph{d}u$ for $u\in \mathcal{A}_{\mathbb{C}}^0(X)$, then $h\mapsto e^{u}\cdot h$ gives an isomorphism $\underline{\mathbb{C}}_{X,\theta}\tilde{\rightarrow}\underline{\mathbb{C}}_X$ of sheaves.

$(2)$ If $\mu$ is a closed $1$-form on $X$, then $\underline{\mathbb{C}}_{X,\theta}\otimes_{\underline{\mathbb{C}}_X}\underline{\mathbb{C}}_{X,\mu}=\underline{\mathbb{C}}_{X,\theta+\mu }$.

$(3)$ Suppose $f:Y\rightarrow X$ is a smooth map between connected smooth manifolds.  Then inverse image sheaf $f^{-1}\underline{\mathbb{C}}_{X,\theta}=\underline{\mathbb{C}}_{Y,f^*\theta}$.
\end{lemma}

\begin{proof}
$(1)$ If $\underline{\mathbb{C}}_{X,\theta}\cong\underline{\mathbb{C}}_{X}$, $H_\theta^0(X)=H^0(X,\underline{\mathbb{C}}_{X,\theta})=\mathbb{C}$. By \cite{HR}, Example 1.6, $\theta$ is exact. Inversely, if $\theta=\textrm{d}u$, $\underline{\mathbb{C}}_{X,\theta}=\underline{\mathbb{C}}e^{-u}$, which implies the conclusion.

$(2)$ Locally, $\theta=\textrm{d}u$ and $\mu=\textrm{d}v$. Then,  $\underline{\mathbb{C}}_{X,\theta}=\mathbb{C}e^{-u}$, $\underline{\mathbb{C}}_{X,\mu}=\mathbb{C}e^{-v}$ and $\underline{\mathbb{C}}_{X,\theta+\mu}=\mathbb{C}e^{-u-v}$, locally. Clearly, the products of functions give an isomorphism $\underline{\mathbb{C}}_{X,\theta}\otimes_{\underline{\mathbb{R}}_X}\underline{\mathbb{C}}_{X,\mu}\rightarrow\underline{\mathbb{C}}_{X,\theta+\mu }$.

$(3)$ Locally, $\theta=\textrm{d}u$, $\underline{\mathbb{C}}_{X,\theta}=\mathbb{C}e^{-u}$ and $\underline{\mathbb{C}}_{Y,f^*\theta}=\mathbb{C}e^{-f^*u}$. The pullbacks of functions give an isomorphism $f^{-1}\underline{\mathbb{C}}_{X,\theta}\tilde{\rightarrow} \underline{\mathbb{C}}_{Y,f^*\theta}$ .
\end{proof}

Let $X$ be a smooth manifold and $\theta$, $\mu$ complex closed  $1$-forms on $X$. The wedge product $\alpha\wedge \beta$ defines a \emph{cup product}
\begin{displaymath}
\cup: H^p_{\theta}(X,\mathbb{C})\times H^q_{\mu}(X,\mathbb{C})\rightarrow H^{p+q}_{\theta+\mu}(X,\mathbb{C}).
\end{displaymath}
Similarly, we can define cup products between $H^p_{\theta}(X,\mathbb{C})$ or $H^p_{\theta,c}(X,\mathbb{C})$ and $H^q_{\mu}(X,\mathbb{C})$ or $H^q_{\mu,c}(X,\mathbb{C})$.

Let $f:X\rightarrow Y$ be a smooth map between connected smooth manifolds and $\theta$ a complex closed $1$-form on $Y$. Set $\tilde{\theta}=f^*\theta$ and $r=\textrm{dim}X-\textrm{dim}Y$.

$(i)$ Define \emph{pullback} $f^*:H^*_{\theta}(Y,\mathbb{C})\rightarrow H^*_{\tilde{\theta}}(X,\mathbb{C})$ as $[\alpha]_{\theta}\mapsto [f^*\alpha]_{\tilde{\theta}}$. If $f$ is proper, we can also define $f^*:H^*_{\theta,c}(Y,\mathbb{C})\rightarrow H^*_{\tilde{\theta},c}(X,\mathbb{C})$ in the same way.

$(ii)$ If $X$ and $Y$ are oriented, define \emph{pushout} $f_*:H^*_{\tilde{\theta},c}(X,\mathbb{C})\rightarrow H^{*-r}_{\theta,c}(Y,\mathbb{C})$ as $[T]_{\theta,c}\mapsto [f_*T]_{\tilde{\theta},c}$. Moreover, if $f$ is proper, $f_*:H^*_{\tilde{\theta}}(X,\mathbb{C})\rightarrow H^{*-r}_{\theta}(Y,\mathbb{C})$ is  defined well similarly.

Let $f:X\rightarrow Y$ be a proper smooth map between connected oriented smooth manifolds. If $\mu$ is a closed $1$-forms on $Y$ and  $\tilde{\theta}=f^*\theta$, we have the \emph{projection formula}
\begin{displaymath}\label{pro-f}
f_*(\sigma\cup f^*\tau)=f_*(\sigma)\cup \tau
\end{displaymath}
for $\sigma$ $\in$ $H_{\tilde{\theta}}^*(X,\mathbb{C})$ or $H_{\tilde{\theta},c}^*(X,\mathbb{C})$ and $\tau\in H_{\mu}^*(Y,\mathbb{C})$ or $H_{\mu,c}^*(Y,\mathbb{C})$. We get it easily by $f_*(T\wedge f^*\beta)=f_*T\wedge \beta$, where $T\in\mathcal{D}^{\prime*}(X)$ and $\beta\in\mathcal{A}^*(Y)$.

Recall that a complex manifold $X$ is called \emph{$p$-K\"ahlerian}, if it admits a closed strictly positive $(p,p)$-form $\Omega$ $($\cite{AB}, Definition 1.1, 1.2$)$. For any $p$-dimensional connected complex submanifold $Z$ of a $p$-K\"ahler manifold $X$, $\Omega|_{Z}$ is a volume form on $Z$. We have
\begin{proposition}\label{inj-surj}
Let $f:X\rightarrow Y$ be a proper surjective holomorphic map between connected complex manifolds, and $\theta$ a complex closed $1$-form on $Y$. Set  $r=\emph{dim}_{\mathbb{C}}X-\emph{dim}_{\mathbb{C}}Y$ and $\tilde{\theta}=f^*\theta$. Assume that $X$ is $r$-K\"ahlerian. Then, for any $p$, $f^*:H_\theta^p(Y,\mathbb{C})\rightarrow H_{\tilde{\theta}}^p(X,\mathbb{C})$ is injective and $f_*:H^{p}_{\tilde{\theta}}(X,\mathbb{C})\rightarrow H^{p-2r}_{\theta}(Y,\mathbb{C})$ is surjective. They also hold for the cases of compact supports.
\end{proposition}
\begin{proof}
Let $\Omega$ be a strictly positive closed $(r,r)$-form on $X$.  Then $c=f_*\Omega$ is a closed  current of degree $0$, hence a constant. By Sard's theorem, the set $U$ of regular values of $f$ is nonempty. For any $y\in U$, $X_y=f^{-1}(y)$ is a $r$-dimensional compact complex submanifold, so $c=\int_{X_y}\Omega|_{X_y}>0$ on $U$.
By the projection formula, $f_*([\Omega]\cup f^*\tau)=c\cdot \tau$, where $[\Omega]\in H^{2r}(X,\mathbb{C})$ and $\tau\in H_{\theta}^{p}(Y,\mathbb{C})$ or $H_{\theta,c}^{p}(Y,\mathbb{C})$. It is easily to deduce the conclusion.
\end{proof}

Clearly, any complex manifold is $0$-K\"ahlerian and any K\"ahler manifold $X$ is  $p$-K\"ahlerian for every $p\leq \textrm{dim}_{\mathbb{C}}X$, so we get
\begin{corollary}\label{inj-surj-1}
Let $f:X\rightarrow Y$ be a proper surjective holomorphic map between connected complex manifolds with the same dimensions. Let $\theta$ be a complex closed $1$-form on $Y$ and $\tilde{\theta}=f^*\theta$. Then, for any $p$, $f^*:H_\theta^p(Y,\mathbb{C})\rightarrow H_{\tilde{\theta}}^p(X,\mathbb{C})$ is injective and $f_*:H^{p}_{\tilde{\theta}}(X,\mathbb{C})\rightarrow H^{p}_{\theta}(Y,\mathbb{C})$ is surjective. They also hold for the cases of compact supports.
\end{corollary}
\begin{corollary}\label{inj-surj-2}
Let $f:X\rightarrow Y$ be a proper surjective holomorphic map between connected complex manifolds and $\theta$ a complex closed $1$-form on $Y$. Set  $r=\emph{dim}_{\mathbb{C}}X-\emph{dim}_{\mathbb{C}}Y$ and $\tilde{\theta}=f^*\theta$. Assume that $X$ is a K\"ahler manifold. Then, for any $p$, $f^*:H_\theta^p(Y,\mathbb{C})\rightarrow H_{\tilde{\theta}}^p(X,\mathbb{C})$ is injective and $f_*:H^{p}_{\tilde{\theta}}(X,\mathbb{C})\rightarrow H^{p-2r}_{\theta}(Y,\mathbb{C})$ is surjective. They also hold for the cases of compact supports.
\end{corollary}

\section{Dolbeault-Morse-Novikov cohomology}
Let $X$ be a $n$-dimensional complex manifold and $\eta$ a $\bar{\partial}$-closed $(0,1)$-form on $X$.  Suppose $\mathcal{A}^{p,q}(X)$ is the space of smooth $(p,q)$-forms on $X$. Define $\bar{\partial}_\eta:\mathcal{A}^{p,q}(X)\rightarrow\mathcal{A}^{p,q+1}(X)$ as follows:
\begin{displaymath}
\bar{\partial}_\eta\alpha=\bar{\partial}\alpha+\eta\wedge\alpha,
\end{displaymath}
for every  $\alpha\in\mathcal{A}^{p,q}(X)$. Clearly, $\bar{\partial}_\eta\circ\bar{\partial}_\eta=0$, so we have a complex
\begin{displaymath}
\xymatrix{
\cdots\ar[r] &\mathcal{A}^{p,q-1}(X)\ar[r]^{\bar{\partial}_\eta}&\mathcal{A}^{p,q}(X)\ar[r]^{\bar{\partial}_\eta}&\mathcal{A}^{p,q+1}(X)\cdots\ar[r]&\cdots
}.
\end{displaymath}
We call its cohomology $H^{p,q}_\eta(X)=H^q(\mathcal{A}^{p,\bullet}(X),\bar{\partial}_\eta)$ \emph{Dolbeault-Morse-Novikov cohomology}. Similarly, we can define \emph{Dolbeault-Morse-Novikov cohomology with compact support} $H^{p,q}_{\eta,c}(X)$. If $\eta=0$, $H^{p,q}_\eta(X)$ is the classical Dolbeault cohomology $H^{p,q}(X)$. Suppose $\mathcal{A}_X^{p,q}$ is the sheaf of germs of  smooth $(p,q)$-forms on $X$. We naturally get a morphism $\bar{\partial}_\eta:\mathcal{A}_X^{p,q}\rightarrow\mathcal{A}_X^{p,q+1}$ of sheaves.
\begin{definition}
We call the kernel of $\bar{\partial}_\eta:\mathcal{A}_X^{p,0}\rightarrow\mathcal{A}_X^{p,1}$ a \emph{weight $\eta$-sheaf of holomorphic $p$-forms}, denoted by $\Omega_{X,\eta}^p$. In particular, $\mathcal{O}_{X,\eta}:=\Omega_{X,\eta}^0$ is called  a weight $\eta$-sheaf of holomorphic functions.
\end{definition}
Locally, by Grothendieck-Poincar\'{e} lemma, $\eta=\bar{\partial}u$ for a smooth complex-valued function $u$, and then, $\bar{\partial}_\eta=e^{-u}\circ \bar{\partial}\circ e^u$. Hence, locally, $\Omega_{X,\eta}^p=e^{-u}\Omega_X^p$, where $\Omega_X^p$ is the sheaf of germs of holomorphic $p$-forms. So $\mathcal{O}_{X,\eta}$ is a locally free sheaf of $\mathcal{O}_X$-modules with rank 1  and
\begin{equation}\label{fundamental}
\Omega_{X,\eta}^p=\Omega_X^p\otimes_{\mathcal{O}_X}\mathcal{O}_{X,\eta}.
\end{equation}
Moreover, we have a soft resolution of $\Omega_{X,\eta}^p$
\begin{displaymath}
\xymatrix{
0\ar[r] &\Omega_{X,\eta}^p\ar[r]^{i} &\mathcal{A}_X^{p,0}\ar[r]^{\bar{\partial}_\eta} &\mathcal{A}_X^{p,q}\ar[r]^{\bar{\partial}_\eta}&\cdots\ar[r]^{\bar{\partial}_\eta}&\mathcal{A}_X^{p,n}\ar[r]&0
}.
\end{displaymath}
Similarly, we can define $\bar{\partial}_\eta$ on the sheaf $\mathcal{D}_X^{\prime p,q}$ of germs of $(p,q)$-currents and have a soft resolution
\begin{displaymath}
\xymatrix{
0\ar[r] &\Omega_{X,\eta}^p\ar[r]^{i} &\mathcal{D}_X^{\prime p,0}\ar[r]^{\bar{\partial}_\eta} &\mathcal{D}_X^{\prime p,1}\ar[r]^{\bar{\partial}_\eta}&\cdots\ar[r]^{\bar{\partial}_\eta}&\mathcal{D}_X^{\prime p,n}\ar[r]&0
}.
\end{displaymath}
So
\begin{displaymath}
H^q(\mathcal{D}^{\prime p,\bullet}(X),\bar{\partial}_\eta)\cong H^{p,q}_\eta(X)\cong H^q(X,\Omega^p_{X,\eta})
\end{displaymath}
and
\begin{displaymath}
H^q(\mathcal{D}_c^{\prime p,\bullet}(X),\bar{\partial}_\eta)\cong H^{p,q}_{\eta,c}(X)\cong H_c^q(X,\Omega^p_{X,\eta}).
\end{displaymath}

Similarly with  Morse-Novikov cohomology, we can define  \emph{pullback} $f^*$, \emph{pushout} $f_*$, \emph{cup product} $\cup$ and  have \emph{projection formulas} on Dolbeault-Morse-Novikov cohomology. Moreover, by the similar proofs of Proposition \ref{inj-surj}, Corollary \ref{inj-surj-1} and \ref{inj-surj-2}, we have
\begin{proposition}
Let $f:X\rightarrow Y$ be a proper surjetive holomorphic map between complex manifolds and $\eta$ a $\bar{\partial}$-closed $(0,1)$-forms on $Y$. Set $r=\emph{dim}_{\mathbb{C}}X-\emph{dim}_{\mathbb{C}}Y$ and $\tilde{\eta}=f^*\eta$. Assume that $X$ is a $r$-K\"{a}hler manifold. Then, for any $p$, $q$, $f^*:H^{p,q}_{\eta}(Y)\rightarrow H^{p,q}_{\tilde{\eta}}(X)$ is injective and $f_*:H^{p,q}_{\tilde{\eta}}(X)\rightarrow H^{p-r,q-r}_{\eta}(Y)$ is surjective. They also hold for the cases of compact supports.
\end{proposition}
\begin{corollary}
Let $f:X\rightarrow Y$ be a proper surjetive holomorphic map between complex manifolds with the same dimensions. Let $\eta$ be a $\bar{\partial}$-closed $(0,1)$-forms on $Y$ and $\tilde{\eta}=f^*\eta$. Then, for any $p$, $q$, $f^*:H^{p,q}_{\eta}(Y)\rightarrow H^{p,q}_{\tilde{\eta}}(X)$ is injective and $f_*:H^{p,q}_{\tilde{\eta}}(X)\rightarrow H^{p,q}_{\eta}(Y)$ is surjective. They also hold for the cases of compact supports.
\end{corollary}
\begin{corollary}
Let $f:X\rightarrow Y$ be a proper surjetive holomorphic map between complex manifolds and $\eta$ a $\bar{\partial}$-closed $(0,1)$-forms on $Y$. Set $r=\emph{dim}_{\mathbb{C}}X-\emph{dim}_{\mathbb{C}}Y$ and $\tilde{\eta}=f^*\eta$. If $X$ is a K\"{a}hler manifold. Then, for any $p$, $q$, $f^*:H^{p,q}_{\eta}(Y)\rightarrow H^{p,q}_{\tilde{\eta}}(X)$ is injective and $f_*:H^{p,q}_{\tilde{\eta}}(X)\rightarrow H^{p-r,q-r}_{\eta}(Y)$ is surjective. They also hold for the cases of compact supports.
\end{corollary}
\begin{remark}
On de Rham and Dolbeault cohomologies, several particular cases were proved in \cite{W}.
\end{remark}

\section{Dolbeault-Morse-Novikov cohomology via sheaf theory}
\subsection{weight $\eta$-sheaf}
First, we give several properties of weight $\eta$-sheaves of holomorphic functions.
\begin{lemma}\label{lem fun 2}
Let $X$ be a complex manifold and $\theta$ a complex closed $1$-form on $X$. Assume $\theta=\bar{\zeta}+\eta$, where $\zeta$ and $\eta$ are the $(0,1)$-forms on $X$. Then

$(1)$ $\mathcal{O}_{X,\eta}=\mathcal{O}_X\otimes_{\underline{\mathbb{C}}_X}\underline{\mathbb{C}}_{X,\theta}$;

$(2)$ $\mathcal{O}_{X,\eta}$, $\mathcal{O}_{X,\zeta}$ and $\underline{\mathbb{C}}_{X,\theta}$ are subsheaves of $\mathcal{A}^0_{X,\mathbb{C}}$. Moreover, $\mathcal{O}_{X,\eta}\cap\overline{\mathcal{O}_{X,\zeta}}=\underline{\mathbb{C}}_{X,\theta}$, where $\overline{\mathcal{O}_{X,\zeta}}$ is the sheaf of complex conjugation of $\mathcal{O}_{X,\zeta}$ in $\mathcal{A}^0_{X,\mathbb{C}}$.
\end{lemma}
\begin{proof}
Locally, $\theta=du$, $\zeta=\bar{\partial}\bar{u}$, $\eta=\bar{\partial}u$, hence, $\underline{\mathbb{C}}_{X,\theta}=\mathbb{C}e^{-u}$, $\mathcal{O}_{X,\eta}=e^{-u}\cdot\mathcal{O}_X$ and $\mathcal{O}_{X,\zeta}=e^{-\bar{u}}\cdot\mathcal{O}_X$. Clearly, $\mathcal{O}_{X,\eta}\cap\overline{\mathcal{O}_{X,\zeta}}=\underline{\mathbb{C}}_{X,\theta}$, and the products of functions give an isomorphism $\mathcal{O}_{X}\otimes_{\underline{\mathbb{C}}_X}\underline{\mathbb{C}}_{X,\theta}\rightarrow\mathcal{O}_{X,\eta}$ .
\end{proof}

\begin{lemma}\label{lem fun 3}
Let $X$ be a complex manifold and $\eta$ a $\bar{\partial}$-closed $(0,1)$-form on $X$.

$(1)$ Suppose $\eta$ is $\bar{\partial}$-exact, i.e., there exists $u\in \mathcal{A}_{\mathbb{C}}^0(X)$, such that $\eta=\bar{\partial}u$ . Then
\begin{displaymath}
\mathcal{O}_{X,\eta}\rightarrow\mathcal{O}_X, h\mapsto h\cdot e^{u}
\end{displaymath}
is an isomorphism of sheaves of $\mathcal{O}_X$-modules.

$(2)$ Suppose $\zeta$ is a $\bar{\partial}$-closed $(0,1)$-form on $X$. Then $\mathcal{O}_{X,\zeta}\otimes_{\mathcal{O}_X}\mathcal{O}_{X,\eta}=\mathcal{O}_{X,\zeta+\eta}$. So $(\mathcal{O}_{X,\eta})^{\vee}=\mathcal{O}_{X,-\eta}$, where $(\mathcal{O}_{X,\eta})^{\vee}=\mathcal{H}om_{\mathcal{O}_X}(\mathcal{O}_{X,\eta},\mathcal{O}_X)$ is the dual of $\mathcal{O}_{X,\eta}$ of $\mathcal{O}_X$-modules.

$(3)$ If $f:Y\rightarrow X$ is a holomorphic map of complex manifolds, then
\begin{displaymath}
f^*\mathcal{O}_{X,\eta}=\mathcal{O}_{Y,f^*\eta},
\end{displaymath}
where $f^*\mathcal{O}_{X,\eta}=f^{-1}\mathcal{O}_{X,\eta}\otimes_{f^{-1}\mathcal{O}_X}\mathcal{O}_Y$ is the inverse image sheaf of $\mathcal{O}_Y$-modules.
\end{lemma}

\begin{proof}
We can get $(1)$ $(2)$ immediately with the similar proof of Lemma \ref{lem fun 1}.

$(3)$ For any presheaf $\mathcal{G}$, denote by $\mathcal{G}^+$ the sheaf  associated to  $\mathcal{G}$. Define presheaves $\mathcal{F}$ and $\mathcal{R}$ on $Y$ as
\begin{displaymath}
\mathcal{F}(U)=\mathop{\underrightarrow{\text{lim}}}\limits_{W\supseteq f(U)}\mathcal{O}_{X,\eta}(W)
\end{displaymath}
and
\begin{displaymath}
\mathcal{R}(U)=\mathop{\underrightarrow{\text{lim}}}\limits_{W\supseteq f(U)}\mathcal{O}_{X}(W),
\end{displaymath}
for any open subset $U$ of $Y$. Then $\mathcal{F}^+=f^{-1}\mathcal{O}_{X,\eta}$, $\mathcal{R}^{+}=f^{-1}\mathcal{O}_{X}$ and $(\mathcal{F}\otimes_{\mathcal{R}}\mathcal{O}_Y)^+=f^*\mathcal{O}_{X,\eta}$.

Define $\varphi(U):\mathcal{F}(U)\otimes_{\mathcal{R}(U)}\mathcal{O}_Y(U)\rightarrow \mathcal{O}_{Y,f^*\eta}(U)$ as $[h]\otimes g\mapsto g\cdot(f^*h)|_U$, for every open subset $U$ of $Y$, where $[h]$ is the class of the $\eta$-holomorphic function $h$ under the direct limit. We get  a morphism $\varphi:\mathcal{F}\otimes_{\mathcal{R}}\mathcal{O}_Y\rightarrow \mathcal{O}_{Y,f^*\eta}$ of presheaves, and moreover, induce a morphism $\varphi^+: f^*\mathcal{O}_{X,\eta}\rightarrow \mathcal{O}_{Y,f^*\eta}$ of sheaves.

We claim that $\varphi^+$ is an isomorphism. Actually, for any $y\in Y$, choose a open ball $V$ near $f(y)$, such that $\eta=\bar{\partial}u$ on $V$ for some $u\in \mathcal{A}_{\mathbb{C}}^0(V)$. The elements of $\mathcal{F}_y=(\mathcal{O}_{X,\eta})_{f(y)}$ and $(\mathcal{O}_{Y,f^*\eta})_{y}$  can be written as $[pe^{-u}]$ and  $[qe^{-f^*u}]$ respectively, where $p$, $q$ are holomorphic functions near $f(y)$, $y$  respectively, where $[a]$ denote the the class of $a$ under direct limit. At the stalk over $y$, $\varphi^+_y([pe^{-u}]\otimes[g])=[g\cdot f^*p\cdot e^{-f^*u}]$, which is isomorphic. We complete the proof.
\end{proof}
\begin{remark}
\emph{If $\eta$ is the $(0,1)$-part of a closed $1$-form, Lemma \ref{lem fun 3} (3) can be proved simply by Lemma \ref{lem fun 2} (1).}
\end{remark}

For a complex closed $1$-form $\theta$ on a complex manifold $X$, we write $\theta=\bar{\zeta}+\eta$, where $\zeta$ and $\eta$ are both $(0,1)$-forms. Let $\partial_{\bar{\zeta}}=\partial+\bar{\zeta}\wedge$. Then $\textrm{d}_\theta=\partial_{\bar{\zeta}}+\bar{\partial}_\eta$, $\partial_{\bar{\zeta}}^2=0$, $\bar{\partial}_{\eta}^2=0$, and $\partial_{\bar{\zeta}}\bar{\partial}_{\eta}+\bar{\partial}_{\eta}\partial_{\bar{\zeta}}=0$. Locally, $\theta=\textrm{d}u$, for a smooth complex-valued function $u$. Then $\eta=\bar{\partial}u$, $\bar{\zeta}=\partial u$ and $\partial_{\bar{\zeta}}=e^{-u}\circ \partial\circ e^u$, locally. By the holomorphic de Rham resolution of $\mathbb{C}$, there exists a resolution of $\underline{\mathbb{C}}_{X,\theta}$
\begin{displaymath}
\xymatrix{
0\ar[r] &\underline{\mathbb{C}}_{X,\theta}\ar[r]^{i} &\mathcal{O}_{X,\eta}\ar[r]^{\partial_{\bar{\zeta}}} &\Omega_{X,\eta}^1\ar[r]^{\partial_{\bar{\zeta}}}&\cdots\ar[r]^{\partial_{\bar{\zeta}}}&\Omega_{X,\eta}^n\ar[r]&0
}.
\end{displaymath}
So we can compute Morse-Novikov cohomology by the hypercohomology  $H_\theta^p(X,\mathbb{C})=\mathbb{H}^p(X,\Omega_{X,\eta}^\bullet)$. If $X$ satisfies that $H_{\eta}^{p,q}(X)=0$ for any $p\geq 1, q\geq 0$, then
\begin{displaymath}
H_\theta^p(X,\mathbb{C})=H^p(\Gamma(X,\Omega_{X,\eta}^\bullet),\partial_{\bar{\zeta}}).
\end{displaymath}
In this case, $H_\theta^p(X,\mathbb{C})=0$ for $p>\textrm{dim}_{\mathbb{C}}X$.

\subsection{K\"{u}nneth formula and Serre's duality}
If $\mathcal{F}$ and $\mathcal{G}$ are sheaves of $\mathcal{O}_X$ and $\mathcal{O}_Y$-modules on complex manifolds $X$ and $Y$ respectively. The \emph{cartesian product sheaf} of $\mathcal{F}$ and $\mathcal{G}$ is defined as
\begin{displaymath}
\mathcal{F}\boxtimes\mathcal{G}=pr_1^*\mathcal{F}\otimes_{\mathcal{O}_{X\times Y}}pr_2^*\mathcal{G},
\end{displaymath}
where $pr_1$ and $pr_2$ are projections from $X\times Y$ onto $X$, $Y$, respectively. Assume that $\zeta$ and $\eta$ are $\bar{\partial}$-closed forms on   complex manifolds $X$ and $Y$ respectively. By the formula $($\ref{fundamental}$)$ and Lemma \ref{lem fun 3} $(3)$,
\begin{displaymath}
pr_1^*\Omega_{X,\zeta}^p=pr_1^*\Omega^p_{X}\otimes_{\mathcal{O}_{X\times Y}}\mathcal{O}_{X\times Y,pr_1^*\zeta}
\end{displaymath}
and
\begin{displaymath}
pr_2^*\Omega_{Y,\eta}^q=pr_2^*\Omega^q_{Y}\otimes_{\mathcal{O}_{X\times Y}}\mathcal{O}_{X\times Y,pr_2^*\eta},
\end{displaymath}
hence $\Omega_{X,\zeta}^p\boxtimes\Omega_{Y,\eta}^q=(\Omega_{X}^p\boxtimes\Omega_{Y}^q)\otimes_{\mathcal{O}_{X\times Y}}\mathcal{O}_{X\times Y,\omega}$, where $\omega=pr_1^*\zeta+ pr_2^*\eta$. So
\begin{equation}\label{dir-sum}
\begin{aligned}
\Omega_{X\times Y,\omega}^k=&\Omega_{X\times Y}^k\otimes_{\mathcal{O}_{X\times Y}}\mathcal{O}_{X\times Y,\omega}\\
=&\left(\bigoplus_{p+q=k}\Omega_{X}^p\boxtimes\Omega_{Y}^q\right)\otimes_{\mathcal{O}_{X\times Y}}\mathcal{O}_{X\times Y,\omega}\\
=&\bigoplus_{p+q=k}\Omega_{X,\zeta}^p\boxtimes\Omega_{Y,\eta}^q.
\end{aligned}
\end{equation}

If $X$ or $Y$  is compact, by $($\ref{dir-sum}$)$ and \cite{Dem}, Chap. IX, (5.23) (5.24), we have an isomorphism
\begin{displaymath}
\bigoplus_{p+q=k,r+s=l}H_\zeta^{p,r}(X)\otimes_{\mathbb{C}} H_\eta^{q,s}(Y)\cong H_\omega^{k,l}(X\times Y)
\end{displaymath}
for any $k$, $l$. We call it \emph{K\"{u}nneth formula} for Dolbeault-Morse-Novikov cohomology.

Let $X$ be a connected compact complex manifold of dimension $n$ and $\eta$ a $\bar{\partial}$-closed $(0,1)$-form on $X$.  By Lemma \ref{lem fun 2}, $(2)$ and Serre duality theorem,
\begin{displaymath}
\cup: H_\eta^{p,q}(X)\times H_{-\eta}^{n-p,n-q}(X)\rightarrow \mathbb{C}
\end{displaymath}
is a nondegenerate pair, for $0\leq p,q\leq n$.

\subsection{Bimeromorphic invariants}
We give several bimeromorphic invariants by Dolbeault-Morse-Novikov cohomology.
\begin{proposition}\label{bim}
Let $f:X\dashrightarrow Y$ be a bimeromorphic map of complex manifolds and $\eta_X, \eta_Y$ $\bar{\partial}$-closed $(0,1)$-forms on $X$, $Y$ respectively. Assume that there exist nowhere dense analytic subsets $E\subseteq X$ and $F\subseteq Y$, such that $f:X-E\rightarrow Y-F$ is biholomorphic and $f^*(\eta_Y|_{Y-F})=\eta_X|_{X-E}$. Then, for any $p$,

$(1)$ $H_{\eta_X}^{0,p}(X)\cong H_{\eta_Y}^{0,p}(Y)$ and $H_{\eta_X,c}^{0,p}(X)\cong H_{\eta_Y,c}^{0,p}(Y)$;

$(2)$ $H_{\eta_X}^{p,0}(X)\cong H_{\eta_Y}^{p,0}(Y)$ and $H_{\eta_X,c}^{p,0}(X)\cong H_{\eta_Y,c}^{p,0}(Y)$.
\end{proposition}

\begin{proof}
We choose two proper modifications $g:Z\rightarrow X$ and  $h:Z\rightarrow Y$ such that there is nowhere dense analytic subset $S$ in $Z$, $E\subseteq g(S)$ and $F\subseteq h(S)$, $g:Z-S\rightarrow X-g(S)$, $h:Z-S\rightarrow Y-h(S)$ are biholomorphic and $fg|_{Z-S}=h|_{Z-S}$. Obviously,
\begin{displaymath}
(g^*\eta_X-h^*\eta_Y)|_{Z-S}=g^*((\eta_X|_{X-E}-f^*(\eta_Y|_{Y-F}))|_{X-g(S)})=0.
\end{displaymath}
By the continuity, $g^*\eta_X=h^*\eta_Y$. Hence, we need only to prove the propostion for the case that $f$ is a proper modification and $f^*\eta_Y=\eta_X$. By \cite{GR2}, page 215, we assume $E=f^{-1}(F)$, $\textrm{codim}_YF\geq 2$ and $\textrm{codim}_XE=1$.

$(1)$ By Lemma \ref{lem fun 3} $(3)$ and \cite{U}, Proposition 1.13, 2.14,
\begin{displaymath}
R^qf_*\mathcal{O}_{X,\eta_X}=R^qf_*\mathcal{O}_X\otimes_{\mathcal{O}_Y}\mathcal{O}_{Y,\eta_Y}=\left\{
 \begin{array}{ll}
\mathcal{O}_{Y,\eta_Y},&~q=0;\\
 &\\
 0,&~otherwise.
 \end{array}
 \right.
\end{displaymath}
Consider Leray spectral sequences,
\begin{displaymath}
E_2^{p,q}=H^p(Y, R^qf_*\mathcal{O}_{X,\eta_X})\Rightarrow H^{p+q}=H^{p+q}(X,\mathcal{O}_{X,\eta_X})
\end{displaymath}
and
\begin{displaymath}
E_2^{p,q}=H_c^p(Y, R^qf_*\mathcal{O}_{X,\eta_X})\Rightarrow H^{p+q}=H_c^{p+q}(X,\mathcal{O}_{X,\eta_X}).
\end{displaymath}
Then $E_2^{p,q}=0$ for $q>0$. Hence $E_2^{p,0}=H^p$. We get $(1)$.

$(2)$ Set $U=X-E$, $V=Y-F$ and $j_U:U\rightarrow X$, $j_V:V\rightarrow Y$ are inclusions.  We have a  commutative diagram
\begin{displaymath}
\xymatrix{
 H^0(Y,\Omega_{Y,\eta_Y}^p)\ar[d]^{j_V^*} \ar[r]^{f^*}& H^0(X,\Omega_{X,\eta_X}^p)\ar[d]^{j_U^*} \\
H^0(V,\Omega_{Y,\eta_Y}^p)       \ar[r]^{(f|_U)^*}& H^0(U,\Omega_{X,\eta_X}^p)     },
\end{displaymath}
By the continuity, the restriction $j_U^*$ is injective. By the second Riemann continuation theorem (\cite{GR1}, p. 133), $j_V^*$ is isomorphic. Since $f|_U$ is biholomorphic, $j_U^*$  is surjective, and then, an isomorphism. So $f^*$ is an isomorphism.

Consider the commutative diagram
\begin{displaymath}
\xymatrix{
 H_c^0(X,\Omega_{X,\eta_X}^p)\ar[d]^{} \ar[r]^{f_*}& H_c^0(Y,\Omega_{Y,\eta_Y}^p)\ar[d]^{} \\
H^0(X,\Omega_{X,\eta_X}^p)       \ar[r]^{f_*}& H^0(Y,\Omega_{Y,\eta_Y}^p) }.
\end{displaymath}
The two vertical maps are inclusions, hence are both injective. We have proven that $f^*:H^{p,0}_{\eta_Y}(Y)\rightarrow H^{p,0}_{\eta_X}(X)$ is an isomorphism. By the projection formula, $f_*f^*=\textrm{id}$ on $H_{\eta_Y}^{p,0}(Y)$. So the map at the bottom is an isomorphism. Then the map at the top is injective. By the projection formula again, $f_*f^*=\textrm{id}$ on $H_{\eta_Y,c}^{p,0}(Y)$, hence $f_*$ is isomorphic on $H_{\eta_X,c}^{p,0}(X)$.
\end{proof}
\begin{remark}
\emph{$H_\theta^1(X,\mathbb{C})$ and $H_{\theta,c}^{2n-1}(X,\mathbb{C})$ are also bimeromorphic invariants, referring to \cite{M1}, Corollary 4.8.}
\end{remark}

\subsection{Leray-Hirsch theorem}
Now we establish the Leray-Hirsch theorem for the Dolbeault-Morse-Novikov cohomology.
\begin{theorem}\label{L-H}
Let $\pi:E\rightarrow X$ be a holomorphic fiber bundle over a connected complex manifold $X$ whose general fiber $F$ is compact and  $\eta$  a $\bar{\partial}$-closed $(0,1)$-form on $X$. Assume there exist classes $e_1,\dots,e_r$ of pure degrees in $ H^{**}(E)$, such that, for every $x\in X$,  their restrictions $e_1|_{E_x},\dots,e_r|_{E_x}$ freely linearly generate $H^{**}(E_x)$. Then, $\pi^*(\bullet)\cup\bullet$ gives isomorphisms of bigraded vector spaces
\begin{displaymath}
H_{\eta}^{**}(X)\otimes_{\mathbb{C}} \emph{span}_{\mathbb{C}}\{e_1, ..., e_r\} \tilde{\rightarrow} H_{\tilde{\eta}}^{**}(E),
\end{displaymath}
where  $\tilde{\eta}=\pi^*\eta$.
\end{theorem}
\begin{proof}
If $X$ is a Stein manifold, the theorem holds. Actually, since $H^{0,1}(X)=0$, $\eta$ is $\bar{\partial}$-exact. By (\ref{fundamental}) and Lemma \ref{lem fun 3} $(1)$, we may assume $\eta=0$. It is exactly \cite{M2}, Theorem 1.2.

Go back to the general case. Let $t_1, ..., t_r$ be forms of pure degrees in $\mathcal{A}^{**}(E)$, such that $e_i=[t_i]$ for $1\leq i\leq r$. Set $L^{*,*}=\textrm{span}_{\mathbb{C}}\{t_1, ..., t_r\}$, which is a bigraded vector spaces and isomorphic to $\textrm{span}_{\mathbb{C}}\{e_1, ..., e_r\}$. For any open set $U$ in $X$, set
\begin{displaymath}
B^{p,q}(U)=\bigoplus_{k+l=p,u+v=q}\mathcal{A}^{k,u}(U)\otimes_{\mathbb{C}} L^{l,v}
\end{displaymath}
and $\bar{\partial}_B=\bar{\partial}_{\eta}\otimes 1$. For any $p$, $(B^{p,\bullet}(U),\bar{\partial}_B)$ is a complex, whose cohomology is
\begin{displaymath}
\begin{aligned}
D^{p,q}(U)=&\left(H_{\eta}^{*,*}(U)\otimes_{\mathbb{C}}\textrm{span}_{\mathbb{C}}\{e_1, ..., e_r\}\right)^{p,q}\\
=&\bigoplus_{k+l=p,u+v=q}H_{\eta}^{k,u}(U)\otimes_{\mathbb{C}}(\textrm{span}_{\mathbb{C}}\{e_1, ..., e_r\})^{l,v}.
\end{aligned}
\end{displaymath}
Clearly, the morphism $\pi^*(\bullet)\wedge\bullet:B^{p,\bullet}(U)\rightarrow C^{p,\bullet}(U):=\mathcal{A}^{p,\bullet}(E_U)$ of complexes induces a morphism on the cohomological level
\begin{displaymath}
\pi^*(\bullet)\cup\bullet:D^{p,q}(U)\rightarrow E^{p,q}(U):=H_{\tilde{\eta}}^{p,q}(E_U),
\end{displaymath}
denoted by $\Phi_U$. We need to prove $\Phi_X$ is an isomorphism.

Given $p$, for any open subsets $U$, $V$ in $X$, there is a commutative diagram of  complexes
\begin{displaymath}
\xymatrix{
0\ar[r]&B^{p,\bullet}(U\cup V)\ar[d]^{\pi^*(\bullet)\wedge\bullet}\quad \ar[r]^{(\rho_U^{U\cup V},\rho_V^{U\cup V})\quad\quad} & \quad B^{p,\bullet}(U)\oplus B^{p,\bullet}(V)\ar[d]^{(\pi^*(\bullet)\wedge\bullet,\pi^*(\bullet)\wedge\bullet)}\quad\ar[r]^{\quad\quad\rho_{U\cap V}^U-\rho_{U\cap V}^V}& \quad B^{p,\bullet}(U\cap V) \ar[d]^{\pi^*(\bullet)\wedge\bullet}\ar[r]& 0\\
0\ar[r]&C^{p,\bullet}(U\cup V)     \quad  \ar[r]^{(j_U^{U\cup V},j_V^{U\cup V})\quad\quad}& \quad C^{p,\bullet}(U)\oplus C^{p,\bullet}(V)    \quad\ar[r]^{\quad\quad j_{U\cap V}^U-j_{U\cap V}^ V} & \quad C^{p,\bullet}(U\cap V)     \ar[r]& 0 },
\end{displaymath}
where $\rho$, $j$ are restrictions and the differentials of complexes in the first, second rows are all $\bar{\partial}_B$, $\bar{\partial}$, respectively. The two rows are both exact sequences of complexes. Therefore, we have a commutative diagram of long exact sequences
\begin{displaymath}
\tiny{\xymatrix{
    \cdots\ar[r]&D^{p,q-1}(U\cap V)\ar[d]^{\Phi_{U\cap V}}\ar[r]&D^{p,q}(U\cup V) \ar[d]^{\Phi_{U\cup V}} \ar[r]& D^{p,q}(U)\oplus D^{p,q}(V) \ar[d]^{(\Phi_U,\Phi_V)}\ar[r]&  D^{p,q}(U\cap V)\ar[d]^{\Phi_{U\cap V}}\ar[r]&\cdots\\
 \cdots\ar[r]&E^{p,q-1}(U\cap V)    \ar[r] & E^{p,q}(U\cup V)\ar[r]&E^{p,q}(U)\oplus E^{p,q}(V)       \ar[r]& E^{p,q}(U\cup V)   \ar[r]&\cdots     .}}
\end{displaymath}
If $\Phi_U$, $\Phi_V$ and $\Phi_{U\cap V}$ are isomorphisms, then $\Phi_{U\cup V}$ is an isomorphism by Five Lemma (seeing \cite{I}, p. 6).  We claim that:

$(\ast)$ For open subsets $U_1,\ldots,U_s\subseteq X$, if $\Phi_{U_{i_1}\cap\ldots \cap U_{i_k}}$ is an isomorphism for any $1\leq k\leq s$ and $1\leq i_1<\ldots<i_k\leq s$, then $\Phi_{\bigcup_{i=1}^s U_i}$ is an isomorphism.

We prove this conclusion by induction. For $r=1$, the conclusion holds clearly.  Suppose it holds for $s$. For $s+1$, set $U'_1=U_1,\ldots, U'_{s-1}=U_{s-1}, U'_s=U_s\cup U_{s+1}$. Then $\Phi_{U'_{i_1}\cap\ldots \cap U'_{i_k}}=\Phi_{U_{i_1}\cap\ldots \cap U_{i_k}}$  is isomorphic for any $1\leq i_1<\ldots<i_k\leq s-1$. Moreover, $\Phi_{U'_{i_1}\cap\ldots \cap U'_{i_{k-1}}\cap U'_s}$  is also isomorphic for any $1\leq i_1<\ldots<i_{k-1}\leq s-1$, since $\Phi_{U_{i_1}\cap\ldots \cap U_{i_{k-1}}\cap U_s}$, $\Phi_{U_{i_1}\cap\ldots \cap U_{i_{k-1}}\cap U_{s+1}}$ and  $\Phi_{U_{i_1}\cap\ldots \cap U_{i_{k-1}}\cap U_s\cap U_{s+1}}$  are isomorphic. By inductive hypothesis, $\Phi_{\bigcup_{i=1}^{s+1} U_i}=\Phi_{\bigcup_{i=1}^{s} U'_i}$ is an isomorphism. We proved $(\ast)$.

For a disjoint union $U=\bigcup U_\alpha$ of open subsets $U_\alpha$ in $X$, $\Phi_U$ is exactly the direct product
\begin{displaymath}
\prod\Phi_{U_\alpha}:\prod D^{p,q}(U_\alpha)\rightarrow \prod H_{\tilde{\eta}}^{p,q}(E_{U_\alpha}).
\end{displaymath}
If $\Phi_{U_\alpha}$ are all isomorphic, then $\Phi_U$ is also an isomorphism.

Let $\mathcal{U}$ be a basis for topology of $X$ such that every $U\in \mathcal{U}$ is Stein and let $\mathcal{U}_\mathfrak{f}$ be  the collection of the finite unions of open sets in  $\mathcal{U}$.

For any finite intersection $V$ of open sets in $\mathcal{U}_\mathfrak{f}$, $\Phi_{V}$ is an isomorphism.  Actually, $V=\bigcap_{i=1}^s U_i$, where $U_i=\bigcup_{j=1}^{r_i}U_{ij}$ and $U_{ij}\in\mathcal{U}$. Then $V=\bigcup_{J\in\Lambda}U_J$, where $\Lambda=\{J=(j_1,...,j_s)|1\leq j_1\leq r_1,\ldots,1\leq j_s\leq r_s\}$ and $U_J=U_{1j_1}\cap...\cap U_{sj_s}$. For any $J_1,\ldots,J_t\in\Lambda$, $U_{J_1}\cap \ldots\cap U_{J_t}$ is a Stein manifold, so $\Phi_{U_{J_1}\cap \ldots\cap U_{J_t}}$ is isomorphic. By $(\ast)$, $\Phi_{V}=\Phi_{\bigcup_{J\in\Lambda}U_J}$ is an isomorphism.

By \cite{GHV}, p. 16, Prop. II, $X=V_1\cup...\cup V_l$, where $V_i$ is a countable disjoint union of open sets in  $\mathcal{U}_\mathfrak{f}$. For any $1\leq i_1<\ldots<i_k\leq l$, $V_{i_1}\cap \ldots \cap V_{i_k}$ is a disjoint union of the finite intersection of open sets in $\mathcal{U}_\mathfrak{f}$. Hence, $\Phi_{V_{i_1}\cap \ldots \cap V_{i_k}}$ is isomorphic, so  is $\Phi_{X}$ by $(\ast)$. We complete the proof.
\end{proof}

In particular, we can calculate the Dolbeault-Morse-Novikov cohomology of projectivized bundles.
\begin{corollary}\label{poj-D}
Let $\pi:\mathbb{P}(E)\rightarrow X$ be the projectivization of a holomorphic vector bundle $E$ on a connected complex manifold $X$. Assume $\eta$ is a $\bar{\partial}$-closed $(0,1)$-form on $X$ and $h=[\frac{i}{2\pi}\Theta(\mathcal{O}_{\mathbb{P}(E)}(-1))]$ is in $H^{1,1}({\mathbb{P}(E)})$, where $\mathcal{O}_{\mathbb{P}(E)}(-1)$ is the universal line bundle  on ${\mathbb{P}(E)}$ and  $\Theta(\mathcal{O}_{\mathbb{P}(E)}(-1))$ is the Chern curvature of a hermitian metric on $\mathcal{O}_{\mathbb{P}(E)}(-1)$. Then $\pi^*(\bullet)\cup\bullet$ gives an isomorphism of graded vector spaces
\begin{displaymath}
H_{\eta}^{*,*}(X)\otimes_{\mathbb{C}}\emph{span}_{\mathbb{C}}\{1,...,h^{r-1}\}\tilde{\rightarrow}H_{\tilde{\eta}}^{*,*}(\mathbb{P}(E)),
\end{displaymath}
where $\emph{rank}_{\mathbb{C}}E=r$ and $\tilde{\eta}=\pi^*\eta$.
\end{corollary}

\subsection{A blow-up formula}
We have the following lemma by definition.
\begin{lemma}[\cite{M2}, Proposition 3.1]\label{com}
Let $X$ be a complex manifold and $Z$, $U$ closed, open complex submanifolds of $X$, respectively. Assume $i:Z\rightarrow X$, $j:U\rightarrow X$, $i':Z\cap U\rightarrow U$ and $j':Z\cap U\rightarrow Z$ are inclusions. Then $i'_*j'^*=j^*i_*$ on $\mathcal{D}^{\prime**}(Z)$.
\end{lemma}

Let $\pi:\widetilde{X}\rightarrow X$ be the blow-up of a connected complex manifold $X$ along a connected complex submanifold $Z$. We know $\pi|_E:E=\pi^{-1}(Z)\rightarrow Z$ is the projectivization $E=\mathbb{P}(N_{Z/X})$ of the normal bundle $N_{Z/X}$. Set
\begin{equation}\label{chern}
h=[\frac{i}{2\pi}\Theta(\mathcal{O}_E(-1))]
\end{equation}
in $H_{\bar{\partial}}^{1,1}(E)$, where $\Theta(\mathcal{O}_E(-1))$ is the curvature of the Chern connection of a hermitian metric of the universal line bundle $\mathcal{O}_E(-1)$ on $E$.
\begin{theorem}\label{main}
With above notations, let $i_E:E\rightarrow\widetilde{X}$ be the inclusion  and $r=\emph{codim}_{\mathbb{C}}Z$. Suppose that $\eta$ is a $\bar{\partial}$-closed $(0,1)$-form on $X$   and $\tilde{\eta}=\pi^*\eta$. Then, for any $p$, $q$,
\begin{equation}\label{b-u-m}
\pi^*+\sum_{i=0}^{r-2}(i_E)_*\circ (h^i\cup)\circ (\pi|_E)^*
\end{equation}
gives an isomorphism
\begin{displaymath}
H_{\eta}^{p,q}(X)\oplus \bigoplus_{i=0}^{r-2}H_{\eta|_Z}^{p-1-i,q-1-i}(Z)\tilde{\rightarrow} H_{\tilde{\eta}}^{p,q}(\widetilde{X}).
\end{displaymath}
\end{theorem}
\begin{proof}
For a Stein manifold $X$, we may assume $\eta=0$ with the same reason with the proof of Theorem \ref{L-H}, so the theorem holds by \cite{M2}, Theorem 1.3.

For the general complex manifold $X$, set
\begin{displaymath}
\mathcal{F}^{p,q}=\mathcal{A}_X^{p,q}\oplus \bigoplus_{i=0}^{r-2}i_{Z*}\mathcal{A}_Z^{p-1-i,q-1-i},
\end{displaymath}
for any $p$, $q$. Define $\bar{\partial}:\mathcal{F}^{p,*}\rightarrow\mathcal{F}^{p,*+1}$ as $(\alpha,\beta_0,...,\beta_{r-2})\mapsto (\bar{\partial}_{\eta}\alpha,\bar{\partial}_{\eta|_Z}\beta_0,...,\bar{\partial}_{\eta|_Z}\beta_{r-2})$. For any $p$, $(\mathcal{F}^{p,\bullet},\bar{\partial})$ is a complex of sheaves.
Let $t=\frac{i}{2\pi}\Theta(\mathcal{O}_E(-1))\in\mathcal{A}^{1,1}(E)$. For any open subset $U$ in $X$, define $\mathcal{F}^{p,q}(U)\rightarrow \mathcal{D}^{\prime p,q}(\widetilde{U})$ as
\begin{displaymath}
\varphi_U=\left\{
 \begin{array}{ll}
(\pi|_{\widetilde{U}})^*+\sum_{i=0}^{r-2}(i_{E\cap \widetilde{U}})_*\circ (t^i|_{E\cap \widetilde{U}}\wedge)\circ (\pi|_{E\cap \widetilde{U}})^*,&~Z\cap U\neq\emptyset\\
 &\\
 (\pi|_{\widetilde{U}})^*,&~Z\cap U=\emptyset,
 \end{array}
 \right.
\end{displaymath}
where $\widetilde{U}=\pi^{-1}(U)$ and $i_{E\cap \widetilde{U}}:E\cap \widetilde{U}\rightarrow \widetilde{U}$ is the inclusion. Clearly, $\bar{\partial}_{\tilde{\eta}}\circ\varphi_U=\varphi_U\circ\bar{\partial}$. Hence, $\varphi_U$ induces a morphism of vector spaces
\begin{displaymath}
\Phi_U:H_{\eta}^{p,q}(U)\oplus \bigoplus_{i=0}^{r-2}H_{\eta|_Z}^{p-1-i,q-1-i}(Z\cap U)\rightarrow H_{\tilde{\eta}}^{p,q}(\widetilde{U}).
\end{displaymath}
We need to prove that $\Phi_X$ is an isomorphism.

For open sets $V\subseteq U$, denote by $\rho^U_V:\mathcal{F}^{p,q}(U)\rightarrow \mathcal{F}^{p,q}(V)$ the restriction of the sheaf $\mathcal{F}^{p,q}$ and $j^U_V:\mathcal{D}^{\prime p,q}(\widetilde{U})\rightarrow \mathcal{D}^{\prime p,q}(\widetilde{V})$ the  restriction of currents. By Lemma \ref{com},  $j^U_V\circ\varphi_U=\varphi_V\circ\rho^U_V$. Given $p$, for any open subsets $U$, $V$ in $X$, there is a commutative diagram of  complexes
\begin{displaymath}
\xymatrix{
0\ar[r]&\mathcal{F}^{p,\bullet}(U\cup V)\ar[d]^{\varphi_{U\cup V}}\quad \ar[r]^{(\rho_U^{U\cup V},\rho_V^{U\cup V})\quad\quad} & \quad\mathcal{F}^{p,\bullet}(U)\oplus \mathcal{F}^{p,\bullet}(V)\ar[d]^{(\varphi_{U},\varphi_{V})}\quad\ar[r]^{\quad\quad\rho_{U\cap V}^U-\rho_{U\cap V}^V}& \quad\mathcal{F}^{p,\bullet}(U\cap V) \ar[d]^{\varphi_{U\cap V}}\ar[r]& 0\\
0\ar[r]&\mathcal{D}^{\prime p,\bullet}(\widetilde{U}\cup \widetilde{V})     \quad  \ar[r]^{(j_U^{U\cup V},j_V^{U\cup V})\quad\quad}& \quad\mathcal{D}^{\prime p,\bullet}(\widetilde{U})\oplus \mathcal{D}^{\prime p,\bullet}(\widetilde{V})    \quad\ar[r]^{\quad\quad j_{U\cap V}^U-j_{U\cap V}^ V} & \quad\mathcal{D}^{\prime p,\bullet}(\widetilde{U}\cap \widetilde{V})     \ar[r]& 0 }.
\end{displaymath}
The two rows are both exact sequences of complexes. For convenience, denote
\begin{displaymath}
L^{p,q}(U)=H_{\eta}^{p,q}(U)\oplus \bigoplus_{i=0}^{r-2}H_{\eta|_Z}^{p-1-i,q-1-i}(Z\cap U).
\end{displaymath}
Therefore, we have a commutative diagram of long exact sequences
\begin{displaymath}
\tiny{\xymatrix{
    \cdots\ar[r]&L^{p,q-1}(U\cap V)\ar[d]^{\Phi_{U\cap V}}\ar[r]&L^{p,q}(U\cup V) \ar[d]^{\Phi_{U\cup V}} \ar[r]& L^{p,q}(U)\oplus L^{p,q}(V) \ar[d]^{(\Phi_U,\Phi_V)}\ar[r]&  L^{p,q}(U\cap V)\ar[d]^{\Phi_{U\cap V}}\ar[r]&L^{p,q+1}(U\cup V)\ar[d]^{\Phi_{U\cup V}}\ar[r]&\cdots\\
 \cdots\ar[r]&H_{\tilde{\eta}}^{p,q-1}(\widetilde{U}\cap \widetilde{V})     \ar[r] & H_{\tilde{\eta}}^{p,q}(\widetilde{U}\cup \widetilde{V})\ar[r]&H_{\tilde{\eta}}^{p,q}(\widetilde{U})\oplus H_{\tilde{\eta}}^{p,q}(\widetilde{V})       \ar[r]& H_{\tilde{\eta}}^{p,q}(\widetilde{U}\cap \widetilde{V})     \ar[r] & H_{\tilde{\eta}}^{p,q+1}(\widetilde{U}\cup \widetilde{V})\ar[r]&\cdots     .}}
\end{displaymath}
Following the steps in the proof of Theorem \ref{L-H}, we proved that $\Phi_X$ is an isomorphism.
\end{proof}

\section{Stability of $\theta$-betti and $\eta$-hodge numbers}
For a  compact smooth manifold $X$ and a real (resp. complex) closed $1$-form $\theta$ on $X$, $b_{k}(X,\theta):=$$\textrm{dim}_{\mathbb{R}}H_\theta^k(X)$ (resp. $\textrm{dim}_{\mathbb{C}}H_\theta^k(X,\mathbb{C})$) is called $k$-th \emph{$\theta$-betti number} of $X$. Similarly, for a compact complex manifold $X$ and a $\bar{\partial}$-closed $(0,1)$-form $\eta$ on $X$, $h_{\eta}^{p,q}(X):=\textrm{dim}_{\mathbb{C}}H_\eta^{p,q}(X)$ is called $(p,q)$-th \emph{$\eta$-hodge number} of $X$.

\begin{lemma}\label{inv}
Let $f:X\rightarrow Y$ be a proper surjective submersion of connected smooth manifolds and $\theta$ a real $($resp. complex$)$ closed $1$-form on $X$. Then, for any $k$, the higher direct image $R^kf_*\underline{\mathbb{R}}_{X,\theta}$  $($resp. $R^kf_*\underline{\mathbb{C}}_{X,\theta}$$)$ is a local system of $\mathbb{R}$ $($resp. $\mathbb{C}$$)$-modules with finite rank.

In particular,
\begin{displaymath}
y\mapsto b_k(X_y,\theta|_{X_y})
\end{displaymath}
is a constant function, where $X_y=f^{-1}(y)$ for any $y\in Y$.
\end{lemma}
\begin{proof}
We may assume $Y$ is an open ball and only prove the real case.

Let $o$ be the center of $Y$. By Ehresmann's trivialization theorem, there exists a diffeomorphism $T:X_o\times Y\rightarrow X$, such that $pr_2=f\circ T$, where $pr_2$ is the projection from $X_o\times Y$ to $Y$. By Lemma \ref{lem fun 1} $(3)$,
\begin{equation}\label{1}
\begin{aligned}
R^kf_*\underline{\mathbb{R}}_{X,\theta}\cong&R^kf_*(T_*\underline{\mathbb{R}}_{X_o\times Y,T^*\theta})\\
\cong&R^k(pr_2)_*\underline{\mathbb{R}}_{X_o\times Y,T^*\theta}.
\end{aligned}
\end{equation}
Set $pr_2$ the projection from $X_o\times Y$ to $X_o$. By K\"{u}nneth formula, $pr_1^*:H^1(X_o)\rightarrow H^1(X_o\times Y)$ is an isomorphism, where we use the fact that $H^0(Y)=\mathbb{R}$ and $H^1(Y)=0$. So, $T^*\theta$ can be written as $pr_1^*\theta_o+\textrm{d}u$ for a closed 1-form $\theta_o$ on $X_o$ and a smooth function $u$ on $X_o\times Y$. Consider the cartesian diagram
\begin{displaymath}
\xymatrix{
 X_o\times Y\ar[d]^{pr_1} \ar[r]^{\quad pr_2}& Y\ar[d]^{p_Y} \\
    X_o               \ar[r]^{p_{X_o}}& \{pt\},}
\end{displaymath}
where $\{pt\}$ is a single point space  and $p_{X_o}$, $p_Y$ are constant map. Evidently, $pr_2$ and $p_{X_o}$ are proper.  By Lemma \ref{lem fun 1} and \cite{I}, p. 316, Corollary 1.5,
\begin{equation}\label{2}
\begin{aligned}
R^k(pr_2)_*\underline{\mathbb{R}}_{X_o\times Y,T^*\theta}\cong&R^k(pr_2)_*\underline{\mathbb{R}}_{X_o\times Y,pr_1^*\theta_o}\\
\cong&R^k(pr_2)_*(pr_1^{-1}\underline{\mathbb{R}}_{X_o,\theta_o})\\
\cong&p_Y^{-1}R^k(p_{X_o})_*(\underline{\mathbb{R}}_{X_o,\theta_o})\\
=&\underline{\mathbb{R}}_{X_o\times Y}\otimes_{\mathbb{R}}H_{\theta_o}^k(X_o).
\end{aligned}
\end{equation}
Combined $($\ref{1}$)$ and $($\ref{2}$)$, $R^kf_*\underline{\mathbb{R}}_{X,\theta}$ is constant on the open ball $Y$. Moreover, the stalk $(R^kf_*\underline{\mathbb{R}}_{X,\theta})_y=H^k(X_y,\underline{\mathbb{R}}_{X,\theta}|_{X_y})=H_{\theta|_{X_y}}^k(X_y)$. We complete the proof.
\end{proof}
Let $X$ be a compact complex manifold and  $\theta=\bar{\zeta}+\eta$ a complex closed $1$-form on $X$, where $\zeta$ and $\eta$ are both $(0,1)$-forms. For the double complex $(\mathcal{A}^{*,*}(X), \partial_{\bar{\zeta}}, \bar{\partial}_\eta)$, the associated simple complex is  $(\mathcal{A}_{\mathbb{C}}^{*}(X), d_\theta)$, which has a natural  filtration
\begin{displaymath}
F^p\mathcal{A}_{\mathbb{C}}^k(X)=\bigoplus_{r\geq p, r+s=k}\mathcal{A}^{r,s}(X).
\end{displaymath}
We get a spectral sequence $(E_r^{*,*},d_r,H^*)$, where $E_1^{p,q}=H_{\eta}^{p,q}(X)$ and $H^k=H_{\theta}^{k}(X,\mathbb{C})$. If $\theta=0$, this is \emph{Fr\"{o}licher spectral sequence}. Clearly, for $p<0$, or $p>n$, or $q<0$, or $q>n$, $E_r^{p,q}=0$. So, for given $p$, $q$,  if $r$ is enough large,
\begin{displaymath}
E_r^{p,q}=E_{r+1}^{p,q}=...=E_\infty^{p,q}=F^pH_{\theta}^{p+q}(X,\mathbb{C})/F^{p+1}H_{\theta}^{p+q}(X,\mathbb{C}).
\end{displaymath}
Since $\textrm{dim}_{\mathbb{C}}E_{r+1}^{p,q}\leq \textrm{dim}_{\mathbb{C}}E_r^{p,q}$ for any $r$,
\begin{displaymath}
b_{k}(X,\theta)=\sum_{p+q=k}E_\infty^{p,q}\leq \sum_{p+q=k}E_1^{p,q}=\sum_{p+q=k}h_{\eta}^{p,q}(X).
\end{displaymath}
The degeneration of this spectral sequence at $E_1$ on compact locally conformally K\"{a}hler manifold is proved in some conditions in \cite{OVV2}.

We say that $f:X\rightarrow Y$ is \emph{a family of complex manifolds}, if $f$ is a proper surjective holomorphic submersion.
\begin{theorem}
Let $f:X\rightarrow Y$ be a family of complex manifolds and $\theta$ a complex closed $1$-form on $X$. Assume $b_k(X_o,\theta|_{X_o})=\sum_{p+q=k}h_{\eta|_{X_o}}^{p,q}(X_o)$ for some $k$ and some point $o\in Y$, where $\eta$ is the $(0,1)$-part of $\theta$. Then, for any $t$ near $o$, $h_{\eta|_{X_t}}^{p,q}(X_t)=h_{\eta|_{X_o}}^{p,q}(X_o)$, where $\eta$ is the $(0,1)$-part of $\theta$ and $p+q=k$.
\end{theorem}
\begin{proof}
Let $\Omega_{X/Y}^1=\Omega_X^1/f^*\Omega_Y^1$ be the sheaf of the relative holomorphic $1$-forms and $\Omega_{X/Y}^p=\bigwedge^p\Omega_{X/Y}^1$. Set $i_t:X_t\rightarrow X$ the inclusion.  Then $i_t^*\Omega_{X/Y}^p=\Omega_{X_t}^p$, seeing \cite{Vo}, p. 234-235.  For the locally free sheaf $\Omega_{X/Y}^p\otimes_{\mathcal{O}_X}\mathcal{O}_{X,\eta}$, we have
\begin{displaymath}
i_t^*(\Omega_{X/Y}^p\otimes_{\mathcal{O}_X}\mathcal{O}_{X,\eta})=i_t^*\Omega_{X/Y}^p\otimes_{\mathcal{O}_{X_t}}i_t^*\mathcal{O}_{X,\eta}=\Omega_{X_t,\eta|_{X_t}}^p.
\end{displaymath}
By the semi-continuity theorem, $h_{\eta|_{X_t}}^{p,q}(X_t)\leq h_{\eta|_{X_o}}^{p,q}(X_o)$ for any $t$ near $o$. So
\begin{displaymath}
b_k(X_o,\theta|_{X_o})=\sum_{p+q=k}h_{\eta|_{X_o}}^{p,q}(X_o)\geq \sum_{p+q=k}h_{\eta|_{X_t}}^{p,q}(X_t)\geq b_k(X_t,\eta|_{X_t}).
\end{displaymath}
By Lemma \ref{inv}, $h_{\eta|_{X_t}}^{p,q}(X_t)=h_{\eta|_{X_o}}^{p,q}(X_o)$ for any $p+q=k$.
\end{proof}

By Hodge decomposition of complex manifolds in Fujiki class $\mathcal{C}$, we get the following corollary immediately.
\begin{corollary}
Let $f:X\rightarrow Y$ be a family of complex manifolds and $\theta$ a complex closed $1$-form on $X$. Assume,  for a point $o\in Y$, $X_o$ is in the Fujiki class $\mathcal{C}$ and $\theta|_{X_o}=0$. Then, for any $t$ near $o$, $h_{\eta|_{X_t}}^{p,q}(X_t)=h^{p,q}(X_o)$, for any $p$, $q$, where $\eta$ is the $(0,1)$-part of $\theta$.
\end{corollary}

\noindent{\bf Acknowledgements.}\,
I would like to express my gratitude to referees for their helpful suggestions and careful reading of my manuscript.

\renewcommand{\refname}{\leftline{\textbf{References}}}

\makeatletter\renewcommand\@biblabel[1]{${#1}.$}\makeatother

\end{document}